\documentclass[11pt]{article}
\usepackage[arrow, matrix]{xy}
\usepackage{amsmath}
\usepackage{cancel}
\usepackage{amssymb}
\usepackage{amsthm}
\usepackage[mathscr]{eucal}
\input{epsf}
\theoremstyle{definition}
\newtheorem{definition}{Definition}
\newtheorem{theorem}[definition]{Theorem}
\newtheorem{proposition}[definition]{Proposition}
\newtheorem{lemma}[definition]{Lemma}

\theoremstyle{remark}
\newtheorem{remark}[definition]{Remark}

\newcounter{enumctr}

\newcommand{\R}{\mathbb{R}}

\newcommand{\C}{\mathbb{C}}
\newcommand{\id}{\hbox{id}}

\newcommand{\eps}{\varepsilon}
\renewcommand{\phi}{\varphi}

\newcommand{\rT}{\mathrm {T}}

\begin{document}
\title{\vspace*{-10mm}
Linearized Asymptotic Stability\\ for Fractional Differential Equations}

\author{
N.D.~Cong\footnote{\tt ndcong@math.ac.vn, \rm Institute of Mathematics, Vietnam Academy of Science and Technology, 18 Hoang Quoc Viet, 10307 Ha
Noi, Viet Nam},
T.S.~Doan\footnote{\tt dtson@math.ac.vn, \rm Institute of Mathematics, Vietnam Academy for Science and Technology, 18 Hoang Quoc Viet, 10307 Ha
Noi, Viet Nam},
S.~Siegmund\footnote{\tt stefan.siegmund@tu-dresden.de, \rm Center for Dynamics, Department of Mathematics, Technische Universit\"{a}t Dresden,
Zellescher Weg 12-14, 01069 Dresden, Germany}
\;and\; H.T.~Tuan\footnote{\tt httuan@math.ac.vn, \rm Institute of Mathematics, Vietnam Academy of Science and Technology, 18 Hoang Quoc Viet,
10307 Ha Noi, Viet Nam}}
\date{}
\maketitle

\begin{abstract}
We prove the theorem of linearized asymptotic stability for fractional differential equations. More precisely, we show that an equilibrium of a
nonlinear Caputo fractional differential equation is asymptotically stable if its linearization at the equilibrium is asymptotically stable. As
a consequence we extend Lyapunov's first method to fractional differential equations by proving that if the spectrum of the linearization is
contained in the sector $\{\lambda \in \C : |\arg \lambda| > \frac{\alpha \pi}{2}\}$ where $\alpha > 0$ denotes the order of the fractional
differential equation, then the equilibrium of the nonlinear fractional differential equation is asymptotically stable.
\end{abstract}

\section{Introduction}
In recent years, fractional differential equations have attracted increasing interest due to the fact that many mathematical problems in science
and engineering can be modeled by fractional differential equations, see e.g.,\ \cite{Kai,Kilbas,Podlubny}.

One of the most fundamental problems in the qualitative theory of fractional differential equations is stability theory. Following Lyapunov's
seminal 1892 thesis \cite{Lyapunov}, two methods are expected to also work for fractional differential equations:

$\bullet$ Lyapunov's First Method: The method of linearization of the nonlinear equation along an orbit, the study of the resulting linear
variational equation by means of Lyapunov exponents (exponential growth rates of solutions), and the transfer of asymptotic stability from the
linear to the nonlinear equation (the so-called theorem of linearized asymptotic stability).

$\bullet$ Lyapunov's Second Method: The method of Lyapunov functions, i.e., of scalar functions on the state space which decrease along orbits.

There have been many publications on Lypunov's second method for fractional differential equations and we refer the reader to
\cite{Li_Chen_Podlubny} or \cite{LiZhang} for a survey.

In this paper we develop Lyapunov's first method for the trivial solution of a fractional differential equation of order $\alpha \in (0,1)$
\begin{equation}\label{Introduction_Eq1}
^{C}D_{0+}^{\alpha}x(t)=Ax(t)+f(x(t)),
\end{equation}
where $A\in \R^{d\times d}$ and $f:\R^d\rightarrow \R^d$ is a continuously differentiable function satisfying that $f(0)=0$ and $Df(0)=0$ (in
fact, we only require a slightly weaker assumption on $f$). The asymptotic stability of (the trivial solution of) its linerization
\begin{equation}\label{Introduction_Eq2}
^{C}D_{0+}^{\alpha}x(t)=Ax(t)
\end{equation}
is known to be equivalent to its spectrum lying in the sector $\{\lambda \in \C : |\arg \lambda| > \frac{\alpha \pi}{2}\}$, see \cite[Theorem.\
7.20]{Kai}. What remains to be shown is that the asymptotic stability of \eqref{Introduction_Eq2} implies the asymptotic stability of the
trivial solution of \eqref{Introduction_Eq1} which is our main result Theorem \ref{main result} on linearized asymptotic stability for
fractional differential equations.

The linearization method is a useful tool in the investigation of stability of equilibria of nonlinear systems: it reduces the problem to a much simpler problem of stability of autonomous linear systems which can be solved explicitly, hence it gives us a criterion for stability of the equilibrium of the nonlinear system. Our theorem does the same service to the investigation of stability of nonlinear fractional differential equations as its classical counterpart does for the investigation of stability of nonlinear ordinary differential equations. 

Note that there are several people dealing with the stability of fractional differential equations similar to our problem: in \cite{Ahmed} our Theorem \ref{main result} is stated but without a complete proof; 
the main literature we are aware of are four papers \cite{Chen,QianLiAgarwalWong2010,Wen08,Zhang15} where the authors formulated a theorem on linearized stability under various assumptions but all these four papers contain serious flaws in the proofs of the theorem which make the proofs incorrect, a detailed discussion can be found in Remark~\ref{discussion}.

The structure of this paper is as follows: In Section 2, we recall some background on fractional calculus and fractional differential equations.
Section 3 is devoted to the main theorem about linear asymptotic stability for fractional differential equations. Section 4 contains an
application of our main result (Theorem \ref{main result}) and discusses a stabilization by linear feedback of a fractional Lotka-Volterra
system. We conclude this introductory section by introducing some notation which is used throughout the paper.

Let $\R^d$ be endowed with the max norm, i.e.,\ $\|x\|=\max(|x_1|,\dots,|x_d|)$ for all $x=(x_1,\dots,x_d)^{\rT}\in\R^d$, let $\R_{\geq 0}$ be
the set of all nonnegative real numbers and $\left(C_\infty(\R_{\geq 0},\R^d),\|\cdot\|_\infty\right)$ denote the space of all continuous
functions $\xi:\R_{\geq 0}\rightarrow \R^d$ such that
\[
\|\xi\|_\infty:=\sup_{t\in \R_{\geq 0}}\|\xi(t)\|<\infty.
\]
It is well known that $\left(C_\infty(\R_{\geq 0},\R^d),\|\cdot\|_\infty\right)$ is a Banach space.
%

%
%

\section{Preliminaries}
%
%
%
We start this section by briefly recalling a framework of fractional calculus and fractional differential equations. We refer the reader to the
books \cite{Kai,Kilbas} for more details.


Let $\alpha>0$ and $[a,b]\subset \R$. Let $x:[a,b]\rightarrow \R$ be a measurable function such that $x\in L^1([a,b])$, i.e.,\
$\int_a^b|x(\tau)|\;d\tau<\infty$. Then, the \emph{Riemann-Liouville integral operator of order $\alpha$} is defined by
\[
I_{a+}^{\alpha}x(t):=\frac{1}{\Gamma(\alpha)}\int_a^t(t-\tau)^{\alpha-1}x(\tau)\;d\tau\quad \hbox{ for } t\in [a,b),
\]
where the \emph{Euler Gamma function } $\Gamma:(0,\infty)\rightarrow \R$ is defined as
\[
\Gamma(\alpha):=\int_0^\infty \tau^{\alpha-1}\exp(-\tau)\;d\tau,
\]
see e.g.,\ \cite{Kai}. The \emph{Caputo fractional derivative } $^{C}D_{a+}^\alpha x$ of a function $x\in C^m([a,b])$, $m:=\lceil\alpha\rceil$
is the smallest integer larger or equal $\alpha$, which was introduced by Caputo (see e.g.,\ \cite{Kai}), is defined by
\[
^{C}D_{a+}^\alpha x(t):=(I_{a+}^{m-\alpha}D^mx)(t),\qquad \hbox{ for } t\in [a,b),
\]
where $D=\frac{d}{dx}$ is the usual derivative. The Caputo fractional derivative of a $d$-dimensional vector-valued function
$x(t)=(x_1(t),\dots,x_d(t))^{\rT}$ is defined component-wise as
$$^{C}D^\alpha_{0+}x(t)=(^{C}D^\alpha_{0+}x_1(t),\dots,^{C}D^\alpha_{0+}x_d(t))^{\rT}.$$
Since $f$ is Lipschitz continuous, \cite[Theorem 6.5]{Kai} implies unique existence of solutions of initial value problems
\eqref{Introduction_Eq1}, $x(0) = x_0$ for $x_0 \in \R^n$. Let $\phi : I \times \R^d \rightarrow \R^d$, $t \mapsto \phi(t,x_0)$, denote the
solution of \eqref{Introduction_Eq1}, $x(0) = x_0$, on its maximal interval of existence $I = [0,t_{\max}(x_0))$ with $0 < t_{\max}(x_0) \leq
\infty$. We now recall the notions of stability and asymptotic stability of the trivial solution of \eqref{Introduction_Eq1}, cf.\
\cite[Definition 7.2, p.~157]{Kai}.
\begin{definition}\label{DS}
The trivial solution of \eqref{Introduction_Eq1} is called:
\begin{itemize}
\item \emph{stable} if for any $\varepsilon >0$ there exists $\delta=\delta(\varepsilon)>0$ such that for every $\|x_0\|<\delta$ we have
$t_{\max}(x_0) =\infty$ and
\[
\|\phi(t,x_0)\|\leq \eps\qquad\hbox{for } t\ge 0.
\]
\item \emph{unstable} if it is not stable.
\item \emph{attractive} if there exists $\widehat{\delta}> 0$ such that $\lim_{t\to \infty}\phi(t,x_0)=0$ whenever $\|x_0\|<\widehat\delta$.
\end{itemize}
The trivial solution is called \emph{asymptotically stable} if it is both stable and attractive.
\end{definition}
For $f=0$, system \eqref{Introduction_Eq1} reduces to a linear time-invariant fractional differential equation
\begin{equation}\label{Eq2}
^{C}D_{0+}^\alpha x(t)=Ax(t).
\end{equation}
As shown in \cite{Kai}, $E_{\alpha}(t^{\alpha}A)x$ solves \eqref{Eq2} with the initial condition $x(0)=x$, where the \emph{Mittag-Leffler matrix
function} $E_{\alpha,\beta}(A)$, for $\beta\in\R$ and a matrix $A\in\R^{d\times d}$ is defined as
\[
E_{\alpha,\beta}(A):=\sum_{k=0}^\infty \frac{A^k}{\Gamma(\alpha k+\beta)},\qquad E_{\alpha}(A):=E_{\alpha,1}(A).
\]
In the following theorem, we recall a spectral characterization on asymptotic stability of the trivial solution of \eqref{Eq2}.
\begin{theorem}
The trivial solution of \eqref{Eq2} is asymptotically stable if and only if
\[
|\hbox{arg}(\lambda)|>\frac{\alpha \pi}{2}\qquad\hbox{for } \lambda\in\sigma(A),
\]
where $\sigma(A)$ is the spectrum of $A$.
\end{theorem}
\begin{proof}
See e.g. \cite[Theorem 7.20]{Kai}.
\end{proof}

In the remaining part of this section, we establish some estimates involving the Mittag-Leffler functions. These estimates will be used to prove
the contraction property of the Lyapunov-Perron operator introduced in the next section. For this purpose, let $\gamma(\varepsilon,\theta)$,
$\varepsilon>0,\,\theta\in (0,\pi]$ denote the contour consisting of the following three parts:
\begin{itemize}
\item [(i)] $\text{arg}(z)=-\theta$, $|z|\ge \varepsilon$,
\item [(ii)] $-\theta\le \text{arg}(z)\le \theta$, $|z|=\varepsilon$,
\item [(iii)] $\text{arg}(z)=\theta$, $|z|\ge \varepsilon$.
\end{itemize}
The contour $\gamma(\varepsilon,\theta)$ divides the complex plane $(z)$ into two domains, which we denote by $G^{-}(\varepsilon,\theta)$ and
$G^{+}(\varepsilon,\theta)$. These domains lie correspondingly on the left and on the right side of the contour $\gamma(\varepsilon,\theta)$.
\begin{lemma}\label{lemma1}
Let $\alpha\in (0,1)$ and $\beta$ be an arbitrary complex number. Then for an arbitrary $\varepsilon>0$ and $\theta\in
(\dfrac{\alpha\pi}{2},\alpha\pi)$, we have
\[
E_{\alpha,\beta}(z)=\dfrac{1}{2\alpha\pi
i}\int_{\gamma(\varepsilon,\theta)}\frac{\exp{(\zeta^{\frac{1}{\alpha}})}\zeta^{\frac{1-\beta}{\alpha}}}{\zeta-z}\,d\zeta\quad\hbox{for all }
z\in G^{-}(\varepsilon,\theta).
\]
\end{lemma}
\begin{proof}
See \cite[Theorem 1.3, p.\ 30]{Podlubny}
\end{proof}
\begin{proposition}\label{Prp1}
Let $\lambda$ be an arbitrary complex number with $\frac{\alpha \pi }{2}< |\text{arg}(\lambda)|\leq \pi$. Then, the following statements hold:
\begin{itemize}
\item[(i)] There exists a positive constant $M(\alpha,\lambda)$ and a positive number $t_0$ such that
\[
|t^{\alpha-1}E_{\alpha,\alpha}(\lambda t^{\alpha})|<\frac{M(\alpha,\lambda)}{t^{\alpha+1}}\quad  \hbox{for any } t>t_0.
\]
\item[(ii)] There exists a positive constant $C(\alpha,\lambda)$ such that
\[
\sup_{t\ge 0}\int_0^t|(t-s)^{\alpha-1}E_{\alpha,\alpha}(\lambda(t-s)^\alpha)|\,ds<C(\alpha,\lambda).
\]
\end{itemize}
\end{proposition}
\begin{proof}
\noindent (i) Note that $\frac{\alpha \pi}{2}< |\text{arg}(\lambda)|\leq \pi$. Hence, there exist $\theta\in (\dfrac{\alpha
\pi}{2},|\text{arg}(\lambda)|)$ and $\theta_0\in (0,\dfrac{\pi\alpha}{2})$ such that $|\text{arg}(\lambda)|- \theta >\theta_0$. Since
$\frac{\alpha\pi}{2}<|\hbox{arg}(\lambda)|\leq \pi$ it follows that $\lambda t^\alpha\in G^{-}(1,\theta+\theta_0)$ for all $t>0$. Thus,
according to Lemma \ref{lemma1} we obtain that
 $$
 E_{\alpha,\alpha}(\lambda t^{\alpha})=
 \dfrac{1}{2\alpha\pi i}\int_{\gamma(1,\theta)}\dfrac{\exp{(\zeta^{\frac{1}{\alpha}})}\zeta^{\frac{1-\alpha}{\alpha}}}{\zeta-\lambda t^{\alpha}}\,d\zeta
 \quad\hbox{for all } t>0.
 $$
Using the identity $\dfrac{1}{\zeta-z}=-\dfrac{1}{z}+\dfrac{\zeta}{z(\zeta-z)}$ leads to
\begin{equation}\label{New_Eq1}
E_{\alpha,\alpha}(\lambda t^{\alpha})= \dfrac{1}{2\alpha\pi
i}\int_{\gamma(1,\theta)}\dfrac{\exp{(\zeta^{\frac{1}{\alpha}})}\zeta^{\frac{1}{\alpha}}}{\lambda t^\alpha (\zeta-\lambda t^{\alpha})}\,d\zeta
\quad\hbox{for all } t>0.
\end{equation}
Let  $t_0:=\dfrac{1}{|\lambda|^{\frac{1}{\alpha}}(1-\sin\theta_0)^{\frac{1}{\alpha}}}$. Then, for all $t\geq t_0$ we have $|\lambda
t^{\alpha}|\geq \frac{1}{1-\sin\theta_0}$. Thus,
\[
|\zeta-\lambda t^{\alpha}|\ge |\lambda t^{\alpha}|\sin \theta_0\quad \hbox{for all } \zeta \in \gamma(1,\theta),
\]
which together with \eqref{New_Eq1} implies that
$$
|E_{\alpha,\alpha}(\lambda t^\alpha)|\le \dfrac{\int_{\gamma(1,\theta)}|\exp{(\zeta^{\frac{1}{\alpha}})}\zeta^{\frac{1}{\alpha}}|d\zeta}{2
\alpha \pi |\lambda|^2 \sin\theta_0}\dfrac{1}{t^{2\alpha}}\quad \hbox{for all } t\geq t_0.
$$
Consequently, for all $t\geq t_0$
$$
|t^{\alpha-1}E_{\alpha,\alpha}(\lambda t^{\alpha})|\le \dfrac{M(\alpha,\lambda)}{t^{\alpha+1}} \quad\hbox{where }
M(\alpha,\lambda):=\dfrac{\int_{\gamma(1,\theta)}|\exp{(\zeta^{\frac{1}{\alpha}})}\zeta^{\frac{1}{\alpha}}|d\zeta}{2\alpha \pi
|\lambda|^2\sin\theta_0}.
$$
\noindent (ii) In what follows, we treat separately two cases $t\leq t_0$ and $t> t_0$, where $t_0$ is defined as in the statement (i).

\noindent \emph{Case 1: $t\leq t_0$:} Note that
\[
\int_0^t s^{\alpha-1}E_{\alpha,\alpha}(\lambda s^\alpha)\;ds=t^\alpha E_{\alpha,\alpha+1}(\lambda t^\alpha),
\]
see, e.g., \cite[pp. 24]{Podlubny}. Therefore, we get that
\begin{eqnarray*}
\int_0^t\left|(t-s)^{\alpha-1}E_{\alpha,\alpha}(\lambda(t-s)^{\alpha})\right|\;ds &\leq&
\int_0^t (t-s)^{\alpha-1}E_{\alpha,\alpha}(|\lambda| (t-s)^\alpha)\,ds\\
&=&t^\alpha E_{\alpha,\alpha+1}(|\lambda| t^\alpha)
\\
&\leq& t_0^\alpha E_{\alpha,\alpha+1}(|\lambda| t_0^\alpha).
\end{eqnarray*}
\noindent \emph{Case 2: $t> t_0$:} From (i), we see that
\begin{equation}\label{Eq13}
\int_0^{t-t_0}\left| (t-s)^{\alpha-1}E_{\alpha,\alpha}(\lambda(t-s)^{\alpha})\right|\;ds \leq
\int_0^{t-t_0}\dfrac{M(\alpha,\lambda)}{(t-s)^{\alpha+1}}\;ds \leq \dfrac{M(\alpha,\lambda)}{\alpha t_0^\alpha}.
\end{equation}
Using a similar statement as in Case 1, we obtain that
\[
\int_{t-t_0}^{t}\left| (t-s)^{\alpha-1}E_{\alpha,\alpha}(\lambda(t-s)^{\alpha})\right|\;ds \leq t_0^\alpha
E_{\alpha,\alpha+1}(|\lambda|t_0^\alpha),
\]
which together with \eqref{Eq13} implies that
\[
\int_{0}^{t} \left|(t-s)^{\alpha-1}E_{\alpha,\alpha}(\lambda(t-s)^{\alpha})\right| \;ds\leq C(\alpha,\lambda),
\]
where  $C(\alpha,\lambda):=\dfrac{M(\alpha,\lambda)}{\alpha t_0^\alpha}+t_0^\alpha E_{\alpha,\alpha+1}(|\lambda| t_0^\alpha)$. The proof is
complete.
\end{proof}
\section{Linearized Asymptotic Stability for Fractional Differential Equations}

We now state the main result of this paper and use the abbreviation $\ell_f(r)$ to denote the Lipschitz constant 
\[
  \ell_f(r):=\sup_{\substack{x,y\in B_{\R^d}(0,r)\\ x\neq y}}\frac{\|f(x)-f(y)\|}{\|x-y\|}
\]
of a locally Lipschitz continuous function $f$ on the ball $B_{\R^d}(0,r):=\{x\in\R^d: \|x\|\leq r\}$.

\begin{theorem}[Linearized Asymptotic Stability for Fractional Differential Equations]\label{main result}
Consider the nonlinear fractional differential equation \eqref{Introduction_Eq1}.  Let $\hat\lambda_1,\dots,\hat\lambda_m$ denote the eigenvalues of $A$
and assume that
\[
|\hbox{arg}(\hat\lambda_i)|>\frac{\alpha \pi}{2},\qquad i=1,\dots,m.
\]
Suppose that the nonlinear term $f:\R^d\rightarrow \R^d$ is a locally Lipschitz continuous function satisfying that
\begin{equation}\label{Lipschitz condition}
f(0)=0,\qquad \lim_{r\to 0}\ell_f(r)=0.
\end{equation}
Then, the trivial solution of \eqref{Introduction_Eq1} is asymptotically stable.
\end{theorem}
Before going to the proof of this theorem, we need two preparatory steps:
\begin{itemize}
\item Transformation of the linear part: The aim of this step is to transform the linear part of \eqref{Introduction_Eq1} to a matrix which is "very close" to a diagonal matrix. This technical step reduces the difficulty in the estimation of the operators constructed in the next step.
\item Construction of an appropriate Lyapunov-Perron operator: In this step, our aim is to present a family of operators with the property that any solution of the nonlinear system \eqref{Introduction_Eq1} can be interpreted as a fixed point of these operators. Furthermore, we show that these operators are contractive and hence the fixed points of these operators can be estimated and can be shown to tend to zero when time goes to infinity.
\end{itemize}
We are now presenting the details of these preparatory steps.
\subsection{Transformation of the linear part}
Using \cite[Theorem 6.37, pp.~146]{Shilov}, there exists a nonsingular matrix $T\in\C^{d\times d}$ transforming $A$ into the Jordan normal form,
i.e.,
\[
T^{-1}A T=\hbox{diag}(A_1,\dots,A_n),
\]
where for $i=1,\dots,n$ the block $A_i$ is of the following form
\[
A_i=\lambda_i\, \id_{d_i\times d_i}+\eta_i\, N_{d_i\times d_i},
\]
where $\eta_i\in\{0,1\}$, $\lambda_i \in \{\hat\lambda_1,\ldots,\hat\lambda_m\}$,  and the nilpotent matrix $N_{d_i\times d_i}$ is given by
\[
N_{d_i\times d_i}:= \left(
      \begin{array}{*7{c}}
      0  &     1         &    0      & \cdots        &  0        \\
        0        & 0    &    1     &   \cdots      &              0\\
        \vdots &\vdots        &  \ddots         &          \ddots &\vdots\\
        0 &    0           &\cdots           &  0 &          1 \\

        0& 0  &\cdots                                          &0         & 0 \\
      \end{array}
    \right)_{d_i \times d_i}.
\]
Let us notice that by this transformation we go from the field of real numbers  out to the field of complex numbers, and we may remain in the field of real numbers only if all eigenvalues of $A$ are real. For a general real-valued matrix $A$ we may simply embed $\R$ into $\C$, consider $A$ as a complex-valued matrix and thus get the above Jordan form for $A$. Alternatively, we may use a more cumbersome real-valued Jordan form (for discussion of a similar issue for FDE see also Diethelm~\cite[pp.~152--153]{Kai}). For simplicity we use the embedding method and omit the discussion on how to return back to the field of real numbers. Note also that this kind of technique is well known in the theory of ordinary differential equations. 

Let $\delta$ be an arbitrary but fixed positive number. Using the transformation $P_i:=\textup{diag}(1,\delta,\dots,\delta^{d_i-1})$, we obtain
that
\begin{equation*}
P_i^{-1} A_i P_i=\lambda_i\, \id_{d_i\times d_i}+\delta_i\, N_{d_i\times d_i},
\end{equation*}
$\delta_i\in \{0,\delta\}$. Hence, under the transformation $y:=(TP)^{-1}x$ system \eqref{Introduction_Eq1} becomes
\begin{equation}\label{NewSystem}
^{C}D_{0+}^\alpha y(t)=\hbox{diag}(J_1,\dots,J_n)y(t)+h(y(t)),
\end{equation}
where $J_i:=\lambda_i \id_{d_i\times d_i}$ for $i=1,\dots,n$ and the function $h$ is given by
\begin{equation}\label{Eq3}
h(y):=\text{diag}(\delta_1N_{d_1\times d_1},\dots,\delta_nN_{d_n\times d_n})y+(TP)^{-1}f(TPy).
\end{equation}
\begin{remark}\label{Remark1}
Note that the map $x\mapsto \text{diag}(\delta_1N_{d_1\times d_1},\dots,\delta_nN_{d_n\times d_n})x$ is a Lipschitz continuous function with
Lipschitz constant $\delta$. Thus, by \eqref{Lipschitz condition} we have
\[
h(0)=0,\qquad \lim_{r\to 0}\ell_h(r)= \left\{
\begin{array}{ll}
\delta & \hbox{if there exists } \delta_i=\delta,\\[1ex]
0 & \hbox{otherwise}.
\end{array}
\right.
\]
\end{remark}
\begin{remark}\label{Remark2}
The type of stability of the trivial solution of equations \eqref{Introduction_Eq1} and \eqref{NewSystem} are the same,
 i.e., they are both stable, attractive or unstable.
\end{remark}
\subsection{Construction of an appropriate Lyapunov-Perron operator}
In this subsection, we concentrate only on  equation \eqref{NewSystem}. We are now introducing a Lyapunov-Perron operator associated with
\eqref{NewSystem}. Before doing this, we discuss some conventions which are used in the remaining part of this section: The space $\R^d$ can be
written as $\R^d=\R^{d_1}\times\dots\times\R^{d_n}$. A vector $x\in\R^d$ can be written component-wise as $x=(x^1,\dots,x^n)$.

For any $x=(x^1,\dots,x^n)\in \R^{d}=\R^{d_1}\times\dots\times \R^{d_n}$, the operator $\mathcal{T}_{x}: C_\infty(\R_{\geq 0},\R^d)\rightarrow
C_\infty(\R_{\geq 0},\R^d)$ is defined by
\[
(\mathcal{T}_{x}\xi)(t)=((\mathcal{T}_{x}\xi)^1(t),\dots,(\mathcal{T}_{x}\xi)^n(t))\qquad\hbox{for } t\in\R_{\geq 0},
\]
where for $i=1,\dots,n$
\begin{eqnarray*}
(\mathcal{T}_{x}\xi)^i(t) = E_\alpha(t^\alpha J_i)x^i+ \int_0^t (t-\tau)^{\alpha-1}E_{\alpha,\alpha}((t-\tau)^\alpha J_i)h^i(\xi(\tau))\;d\tau,
\end{eqnarray*}
is called the \emph{Lyapunov-Perron operator associated with \eqref{NewSystem}}. The role of this operator is stated in the following theorem.
\begin{theorem}\label{Var_Const_Form}
Let $x\in\R^{d}$ be arbitrary and $\xi:\R_{\geq 0}\rightarrow \R^{d}$ be a continuous function satisfying that $\xi(0)=x$. Then, the following
statements are equivalent:
\begin{itemize}
\item [(i)] $\xi$ is a solution of \eqref{NewSystem} satisfying the initial condition $x(0)=x$.
\item [(ii)] $\xi$ is a fixed point of the operator $\mathcal {T}_{x}$.
\end{itemize}
\end{theorem}
\begin{proof}
The assertion follows from the variation of constants formula for fractional differential equations, see e.g., \cite{Kilbas}.
\end{proof}

Next, we provide some estimates on the operator $\mathcal{T}_{x}$. The main ingredient to obtain these estimates is the preparatory work in
Proposition \ref{Prp1}.
\begin{proposition}\label{Prp2} Consider system \eqref{NewSystem} and suppose that
\[
|\hbox{arg}(\lambda_i)|>\frac{\alpha \pi}{2},\qquad i=1,\dots,n.
\]
Then, there exists a constant $C(\alpha,\mathbf{\lambda})$ depending on $\alpha$ and
$\lambda:=(\lambda_1,\dots,\lambda_n)$ such that for all $x,\widehat x\in\R^{d}$ and $\xi,\widehat\xi\in C_\infty(\R_{\geq
0},\R^d)$ the following inequality holds
\begin{eqnarray*}\label{Contraction}
\notag \|\mathcal T_{x}\xi-\mathcal T_{\widehat x}\widehat\xi\|_\infty & \leq &
\max_{1\le i \le n}\sup_{t\geq 0}|E_\alpha(\lambda_i t^\alpha)| \|x-\widehat x\|\\
&+&C(\alpha,\lambda)\; \ell_h(\max(\|\xi\|_\infty,\|\widehat\xi\|_\infty))\|\xi-\widehat\xi\|_\infty.
\end{eqnarray*}
Consequently, $\mathcal T_{x}$ is well-defined and
\begin{equation}\label{well_defined}
\|\mathcal T_{x}\xi-\mathcal T_{x}\widehat\xi\|_\infty \leq C(\alpha,\lambda)\;
\ell_h(\max(\|\xi\|_\infty,\|\widehat\xi\|_\infty))\|\xi-\widehat\xi\|_\infty.
\end{equation}
\end{proposition}
\begin{proof}
For $i=1,\dots,n$, we get
\begin{align*}
\|(\mathcal{T}_x \xi)^i(t)-(\mathcal{T}_{\widehat{x}} \widehat{\xi})^i(t)\| \le &
\;\|x-\widehat x\||E_\alpha (\lambda_i t^\alpha)| + \\
&\hspace{-30mm} \ell_h(\max\{\|\xi\|_\infty,\|\widehat{\xi}\|_\infty\}) \|\xi-\widehat{\xi}\|_\infty \int_0^t|(t-\tau)^{\alpha-1}
E_{\alpha,\alpha}(\lambda_i (t-\tau)^\alpha)|\;d\tau.
\end{align*}
According to Proposition \ref{Prp1}(ii), we have
\begin{align*}
\|(\mathcal T_x \xi)^i-(\mathcal T_{\widehat x} \widehat \xi)^i\|_\infty  \le & \;\|x-\widehat x\|\sup_{t\ge 0}|E_\alpha (\lambda_i t^\alpha)|\\
& \hspace{-15mm}+ \ell_h(\max\{\|\xi\|_\infty,\|\widehat{\xi}\|_\infty\})C(\alpha,\lambda_i)\|\xi-\widehat{\xi}\|_\infty.
\end{align*}
Letting $C(\alpha,\mathbf{\lambda})=\max\{C(\alpha,\lambda_1),\dots,C(\alpha,\lambda_n)\}$, we obtain the estimate
\begin{align*}
\|\mathcal T_{x}\xi-\mathcal T_{\widehat x}\widehat\xi\|_\infty \leq &\;
\max_{1\le i \le n}\sup_{t\geq 0}|E_\alpha(\lambda_i t^\alpha)| \|x-\widehat x\|\\
& \hspace{-10mm}+C(\alpha,\lambda)\; \ell_h(\max(\|\xi\|_\infty,\|\widehat\xi\|_\infty))\|\xi-\widehat\xi\|_\infty,
\end{align*}
which leads to
\[
\|\mathcal T_{x}\xi-\mathcal T_{x}\widehat\xi\|_\infty \leq C(\alpha,\lambda)\;
\ell_h(\max(\|\xi\|_\infty,\|\widehat\xi\|_\infty))\|\xi-\widehat\xi\|_\infty.
\]
Note that from the definition of the Lyapunov-Perron operator $ \mathcal T_x $, $ \mathcal T_0(0)=0$. The proof is complete.
\end{proof}
So far, we have proved that the Lyapunov-Perron operator is well-defined and Lipschitz continuous. Note that the Lipschitz constant
$C(\alpha,\lambda)$ is independent of the constant $\delta$ which is hidden in the coefficients of system \eqref{NewSystem}. From now on, we
choose and fix the constant $\delta$ as follows $\delta:=\frac{1}{2C(\alpha,\lambda)}$. The remaining difficult question is now to choose a ball
with small radius in $C_\infty(\R_{\geq 0},\R^d)$ such that the restriction of the Lyapunov-Perron operator to this ball is
strictly contractive. A positive answer to this question is given in the following technical lemma.
\begin{lemma}\label{lemma6}
The following statements hold:
\vspace{-5mm}
\begin{itemize}
\item [(i)] There exists $r>0$ such that
\begin{equation}\label{Eq7a}
q:=C(\alpha,\lambda)\;  \ell_h(r) < 1.
\end{equation}
\item [(ii)] Choose and fix $r>0$ satisfying \eqref{Eq7a}. Define
\begin{equation}\label{Eq7b}
r^*:=\frac{r(1-q)}{\max_{1\le i\le n}\sup_{t\ge 0}|E_\alpha(\lambda_it^\alpha)|}.
\end{equation}
Let  $B_{C_{\infty}}(0,r):=\{\xi\in C_\infty(\R_{\geq 0},\R^d):\left||\xi|\right|_\infty\le r\}$. Then, for any $x\in
B_{\R^{d}}(0,r^*)$ we have $\mathcal T_{x} (B_{C_{\infty}}(0,r))\subset B_{C_{\infty}}(0,r)$ and
\begin{equation*}\label{LipschitzContinuity}
\|\mathcal T_{x}\xi-\mathcal T_{x}\widehat {\xi}\|_\infty \leq q\|\xi-\widehat{\xi}\|_\infty\quad\hbox{ for all } \xi,\widehat{\xi}\in
B_{C_{\infty}}(0,r).
\end{equation*}
\end{itemize}

\end{lemma}
\begin{proof}
(i) By Remark \ref{Remark1}, $\lim_{r\to 0}\ell_h(r)\leq \delta$. Since $\delta C(\alpha,\lambda)=\frac{1}{2}$, the assertion (i) is proved.

(ii) Let $x\in \R^{d}$ be arbitrary with $\|x\|\leq r^*$. Let $\xi\in B_{C_{\infty}}(0,r)$ be arbitrary. According to \eqref{Contraction} in
Proposition \ref{Prp2}, we obtain that
\begin{eqnarray*}
\|\mathcal T_{x}\xi\|_\infty
& \leq & \;\max_{1\le i \le n}\sup_{t\ge 0}|E_\alpha(\lambda_it^\alpha)|\|x\|+ C(\alpha,\lambda)\,\ell_{h}(r)\|\xi\|_{\infty}\\
& \leq& \;(1-q)r+qr,
\end{eqnarray*}
which proves that $\mathcal T_{x}(B_{C_{\infty}}(0,r))\subset B_{C_{\infty}}(0,r)$. Furthermore, by  Proposition \ref{Prp1} and part (i) for all
$x\in B_{\R^d}(0,r^*)$ and $\xi,\widehat{\xi}\in B_{C_\infty}(0,r)$ we have
\begin{eqnarray*}
\|\mathcal T_{x}\xi-\mathcal T_{x}\widehat{\xi}\|_\infty &\leq&
C(\alpha,\lambda)\ell_{h}(r)\;\|\xi-\widehat{\xi}\|_\infty\\[1.5ex]
&\leq & q\|\xi-\widehat{\xi}\|_\infty,
\end{eqnarray*}
which concludes the proof.
\end{proof}
\begin{proof}[Proof of Theorem \ref{main result}]
Due to Remark \ref{Remark2}, it is sufficient to prove the asymptotic stability for the trivial solution of system \eqref{NewSystem}. For this
purpose, let $r^*$ be defined as in \eqref{Eq7b}. Let $x\in B_{\R^d}(0,r^*)$ be arbitrary. Using Lemma \ref{lemma6} and the Contraction Mapping
Principle, there exists a unique fixed point $\xi \in B_{C_\infty}(0,r)$ of $\mathcal{T}_x$. This point is also a solution of
\eqref{NewSystem} with the initial condition $\xi(0)=x$. Since the  initial value problem for Equation
\eqref{NewSystem} has unique solution,  this shows that the trivial solution $0$ is stable. To complete the proof of the theorem, we have to show that the trivial solution $0$ is
attractive. Suppose that $\xi(t)=((\xi)^1(t),\dots,(\xi)^n(t))$ is the solution of \eqref{NewSystem} which satisfies $\xi(0)=x$ for
an arbitrary $x=(x^1,\dots,x^n)\in B_{\R^d}(0,r^*)$. From Lemma \ref{lemma6}, we see that $\|\xi\|_\infty\le r$. Put
$a:=\limsup_{t\to\infty}\|\xi(t)\|$, then $a\in [0,r]$. Let $\varepsilon$ be an arbitrary positive number. Then, there exists $T(\varepsilon)>0$
such that
\[
\|\xi(t)\|\le (a+\varepsilon)\qquad \textup{for any}\;t\ge T(\varepsilon).
\]
For each $i=1,\dots,n$, we will estimate $\limsup_{t\to\infty}\|(\xi)^i(t)\|$. According to Proposition \ref{Prp1}(i), we obtain
\begin{eqnarray*}
&&\limsup_{t\to\infty}\left|\int_0^{T(\varepsilon)}(t-\tau)^{\alpha-1}E_{\alpha,\alpha}(\lambda_i(t-\tau)^\alpha)h^i(\xi(\tau))d\tau\right|\\[1.5ex]
&\le& \max_{t\in [0,T(\varepsilon)]}\|h^i(\xi(t))\|\limsup_{t\to \infty}\int_0^{T(\varepsilon)}\frac{M(\alpha,\lambda_i)}{(t-\tau)^{\alpha+1}}d\tau\\
&= &0.
\end{eqnarray*}
Therefore, from the fact that $(\xi)^i(t)=(\mathcal{T}_x \xi)^i(t)$ and $\lim_{t\to \infty}E_\alpha(\lambda_it^\alpha)=0$ we have
\begin{eqnarray*}
\limsup_{t\to\infty}\|(\xi)^i(t)\| &=&
\limsup_{t\to\infty}\left|\int_{T(\varepsilon)}^t(t-\tau)^{\alpha-1}E_{\alpha,\alpha}(\lambda_i(t-\tau)^\alpha)h^i(\xi(\tau))d\tau\right|\\
&\le& \ell_h(r)C(\alpha,\lambda_i)(a+\varepsilon),
\end{eqnarray*}
where we use the estimate
\begin{eqnarray*}
\left|\int_{T(\varepsilon)}^t(t-\tau)^{\alpha-1}E_{\alpha,\alpha}(\lambda_i(t-\tau)^\alpha)d\tau\right| &=&
\left|\int_0^{t-T(\varepsilon)}u^{\alpha-1}E_{\alpha,\alpha}(\lambda_iu^\alpha)du\right|\\
&\le& C(\alpha,\lambda_i),
\end{eqnarray*}
see Proposition \ref{Prp1}(ii), to obtain the inequality above. Thus,
\begin{align*}
a &\le \max\left\{\limsup_{t\to\infty}\|(\xi)^1(t)\|,\dots,\limsup_{t\to\infty}\|(\xi)^n(t)\|\right\}\\
& \le  \ell_h(r)C(\alpha,\lambda)(a+\varepsilon).
\end{align*}
Letting $\varepsilon\to 0$, we have
\[
a \le \ell_h(r)C(\alpha,\lambda)a.
\]
Due to the assumption $\ell_h(r)C(\alpha,\lambda)<1$, we get that $a=0$ and the proof is complete.
\end{proof}
\begin{remark}[Discussion about some related papers]\label{discussion}
As mentioned at the beginning of this paper there are some papers dealing with the problem of linearized stability of fractional differential equations  \cite{Chen,QianLiAgarwalWong2010,Wen08,Zhang15} where the authors formulated a theorem on linearized stability under various assumptions. Here we show that these papers \cite{Chen,QianLiAgarwalWong2010,Wen08,Zhang15} contain serious flaws in the proofs of the linearized stability theorem which make the proofs incorrect. Namely, there are two common flaws in those papers:
\begin{itemize}
\item {\em Incorrect application of the Gronwall lemma:} The authors apply the Gronwall lemma to get an estimate of a solution of the fractional differential equation under consideration (see \cite[l.~1, p.~604]{Chen}, \cite[l.~-8, p.~869]{QianLiAgarwalWong2010}, \cite[l.~6, column~2, p.~1180]{Wen08} and \cite[l.~-6, column~2, p.~103]{Zhang15}), but the multiplier function in the inequality they want to apply the Gronwall lemma to {\em does depend on the variable $t$} besides the variable $\tau$ of the integration. This circumstance makes their application of the Gronwall lemma invalid.
\item  {\em Invalid assumption of smallness of the solution:} The authors of \cite{Chen,Wen08,Zhang15} need the assumption of smallness of the solution $x(t)$ of the nonlinear system for all $t$ (see \cite[formulas (13) and (14), p.~603]{Chen},  \cite[formulas (23) and (26), p.~1180]{Wen08} and \cite[l.~-9, column~2, p.~103]{Zhang15}). Note that the smallness of $x(t)$ for all $t$ is a claim that must be proved in this case and the authors did not prove it at all. Moreover, this claim, in some sense, is almost equivalent to the conclusion about stability of the nonlinear system they wanted to prove.
\end{itemize}
For the paper \cite{QianLiAgarwalWong2010} (dealing with the Riemann--Liouville fractional derivative), since they first treated the case of  linear perturbation \cite[Theorem~4.1]{QianLiAgarwalWong2010}, they did not encounter the second flaw above, but with the first flaw they did arrive at wrong assertions in their theorems both in the linear case \cite[Theorem~4.1(a,b)]{QianLiAgarwalWong2010} as well as the nonlinear case \cite[Theorem~4.2(a,b)]{QianLiAgarwalWong2010}. An easy counterexample for the linear case  \cite[Theorem 4.1(a,b)]{QianLiAgarwalWong2010}  is $B=I-A$ with $I$ being the identity matrix.
\end{remark}

\section{Applications}
In this section, we revisit the problem of stablization by linear feedback of the following fractional Lotka-Volterra system:
\begin{align}\label{Lotka}
\begin{cases}
^{C}D^{\alpha}_{0+}x_1(t)&=x_1(t)(h+ax_1(t)+bx_2(t)),\\[0,5ex]
^{C}D^{\alpha}_{0+}x_2(t)&=x_2(t)(-r+cx_1(t)),\\[0.5ex]
\end{cases}
\end{align}
where the parameters $h,r$ are positive, see e.g.,\ \cite{Ahmed,Wen08}. This system can be rewritten as follows
\begin{equation}\label{FLS}
^{C}D^{\alpha}_{0+}x(t)=Ax(t)+f(x(t)),
\end{equation}
where
\begin{align*}
A=\begin{bmatrix}
h&0\\
0&-r
\end{bmatrix},
&& f(x)=\begin{bmatrix} ax_1^2+bx_1x_2\\cx_1x_2
\end{bmatrix}.
\end{align*}
In the following lemma, we first prove instability of the trivial solution for system \eqref{Lotka}. Finally, we show that by using a suitable
state-feedback controller, the controlled system becomes stable.
\begin{lemma}
The following statements hold:
\begin{itemize}
\item [(i)] The trivial solution of \eqref{Lotka} is unstable.
\item [(ii)] Letting  $B=(1,1)^{\rT}$ and $K=(-2h,0)$. Then, the trivial solution of the following closed-loop system
\begin{align}
^{C}D^\alpha_{0+}x(t)&=Ax(t)+f(x(t))+Bu(t),\notag\\[0.5ex]
u(t) &=K x(t),\notag
\end{align}
is stable.
\end{itemize}
\end{lemma}
\begin{proof}
(i) Choose and fix an arbitrary positive number $\eps$ such that $ \varepsilon |a| < \frac{h}{2} $. Suppose to the contrary that the trivial
solution of \eqref{Lotka} is stable. Then, there exists $\delta\in (0,\varepsilon)$ such that for any solution $(x_1(t),x_2(t))^{\rT}$ of
\eqref{Lotka} with the initial value satisfying $|x_1(0)|+|x_2(0)|<\delta$, then $|x_1(t)|+|x_2(t)|<\varepsilon$ for every $t\geq 0$. We now consider
the solution $(x_1(t),x_2(t))^{\rT}$ of \eqref{Lotka} satisfying that $x_1(0)=\frac{\delta}{2}$ and $x_2(0)=0$. From \eqref{Lotka} and
$x_2(0)=0$, we have $x_2(t)=0$ for all $t\geq 0$. Let $[0,T_{\rm max}]$ denote the maximal interval on which the solution $x_1(t)$ is nonnegative.
Since $ \varepsilon |a|< \frac{h}{2} $, it follows that
\[
^{C}D^{\alpha}_{+0}x_1(t)  \ge \dfrac{h}{2}x_1(t)\qquad\hbox{for all } t\in [0,T_{\rm max}].
\]
By \cite[Lemma 6.1]{Li_Chen_Podlubny2010}, we have
\begin{align*}
x_1(t) \ge E_\alpha(\displaystyle\dfrac{h}{2}t^\alpha)x_1(0)\qquad\hbox{for all } t\in [0,T_{\rm max}].
\end{align*}
Using continuity of the map $t\mapsto x_1(t)$, we obtain that $T_{\rm max}=\infty$ and therefore $x_1(t)\geq E_\alpha(\dfrac{h}{2}t^\alpha)x_1(0)$ for all $t\geq
0$. This contradicts the fact that $ \lim_{t\to\infty}E_\alpha(\displaystyle\dfrac{h}{2}t^\alpha)= \infty $. The proof of this part is complete.

\noindent (ii) The linear part of the closed-loop system is
\[
A+BK=\begin{bmatrix}
-h&0\\
-2h&-r
\end{bmatrix},
\]
which implies that the eigenvalues of $A+BK$ are $-h$ and $-r$. According to Theorem \ref{main result}, the zero solution of the closed-loop
system is asymptotically stable for any order $\alpha\in (0,1)$.
\end{proof}
\section*{Acknowledgement}
This research of the first, the second and the fourth author is funded by the Vietnam National Foundation for Science and Technology Development (NAFOSTED) under Grant Number 101.03-2014.42.

\end{document}